\documentclass[journal,twoside,web]{IEEEcolor}
\usepackage{generic}
\usepackage{stfloats}
\usepackage{savesym}
\usepackage{amsmath}
\usepackage{textcomp}
\savesymbol{iint}
\restoresymbol{TXF}{iint}
\usepackage{graphicx,subfigure}
\usepackage{epstopdf}
\usepackage{amssymb,amsfonts}
\usepackage{cite}
\usepackage{latexsym}
\usepackage{psfrag}
\usepackage{pifont}
\usepackage{caption}
\usepackage{color}
\usepackage{tikz}
\usepackage{bm}
\usetikzlibrary{arrows,shapes,chains}
\usepackage{bbm,mathrsfs}%
\usepackage{fancyhdr} % 添加页眉页脚
\usepackage{verbatim} % 添加注释
\usepackage{xcolor}
\usepackage{setspace}
\usepackage{algorithm}
\usepackage{algpseudocode}
\usepackage{amsmath}
\let\labelindent\relax
\usepackage{enumitem}
\newtheorem{theorem}{Theorem}

\newtheorem{example}{Example}
\newtheorem{remark}{Remark}

\newtheorem{definition}{Definition}

\def\BibTeX{{\rm B\kern-.05em{\sc i\kern-.025em b}\kern-.08em
		T\kern-.1667em\lower.7ex\hbox{E}\kern-.125emX}}
\begin{document}
\title{Distributed Fault Diagnosis in Discrete Event Systems with Transmission Delay Impairments}

\author{~Jiwei Wang,~Simone Baldi,~\IEEEmembership{Senior Member,~IEEE},~Wenwu Yu,~\IEEEmembership{Senior Member,~IEEE},\\~Xiang Yin,~\IEEEmembership{Member,~IEEE}\vspace{-1em}
	\thanks{This work was supported in part by the National Natural Science Foundation of China under Grants 61673107, 62073074, 62073076; in part by the Key Intergovernmental Special Fund of National Key Research and Development Program Grant 2022YFE0198700; and in part by the Research Fund for International Scientists Grant 62150610499. (Corresponding authors: Simone Baldi, Wenwu Yu)}
	\thanks{J. Wang, S. Baldi and W. Yu are with School of Cyber Science and Engineering, Southeast University, Nanjing 210096, China (e-mails: jwwang@seu.edu.cn; S.Baldi@tudelft.nl; wwyu@seu.edu.cn).}
	\thanks{X. Yin is with the Department of Automation and Key Laboratory of System Control and Information Processing, Shanghai Jiao Tong University, Shanghai 200240, China (e-mail: yinxiang@sjtu.edu.cn).}
}

\maketitle

\begin{abstract}
	This note studies the distributed fault diagnosis problem in partially-observed discrete event systems, where the system is monitored by a group of agents to cooperatively diagnose faults within a finite number of steps. 
	The novelty of this work is the creation of a methodology to verify when the faults can be diagnosed even in the presence of transmission delay impairments.
	To address this scenario, a new distributed diagnosability condition is proposed, which extends decentralized diagnosability conditions proposed in the literature.
	Such distributed diagnosability condition is then verified via a novel structure named delay recorder and a new diagnosis function. 
	Theoretical analysis shows that the verification method can successfully determine whether the faults can be diagnosed.
\end{abstract}
\vspace*{-1em}

\IEEEpeerreviewmaketitle

\section{Introduction}\label{1}

In recent decades, fault diagnosis for discrete event systems (DES) has attracted increasing attention \cite{sampath1995diagnosability,cassandras2009introduction,lafortune2018history}, with the most studied problems being the verification \cite{yin2019robust,keroglou2019verification} and the synthesis problems \cite{wang2019optimizing,hu2020design}.
In fault diagnosis of DES, the challenge is to diagnose the occurrence of a fault in finite steps by only observing limited events, while other events including the faults are unobservable. 
In such partially-observed DES, the fault diagnosis architecture is called centralized when there is only one agent monitoring all observable events \cite{yin2019general,white2019fault}.

However, as limited coverage and limited communication ability may make a centralized architecture unpractical, multi-agent architectures have been proposed, where the system is monitored by a group of agents, each one having partial observation capability \cite{debouk2000coordinated,yin2015codiagnosability,viana2019codiagnosabilitya,nunes2018codiagnosability,viana2021codiagnosability,keroglou2018distributed,veras2021distributed}.
If each agent can share its information with a few neighboring agents, the multi-agent architecture is called distributed, otherwise it is decentralized. 
Decentralized multi-agent diagnosis problems have been mostly considered in DES literature, starting from the notion of codiagnosability \cite{debouk2000coordinated}: codiagnosability is an extension of the single-agent (centralized) diagnosability to a multi-agent decentralized scenario, i.e. without information sharing among agents.
Related notions have been studied:
\cite{yin2015codiagnosability} discussed the connection between codiagnosability and coobservability;
\cite{viana2019codiagnosabilitya} revisited codiagnosability with a new necessary and sufficient condition; in \cite{nunes2018codiagnosability,viana2021codiagnosability,keroglou2018distributed,veras2021distributed}, different notions of codiagnosability and verification methods were proposed, where \cite{keroglou2018distributed} and \cite{veras2021distributed} studied distributed diagnosis under ideal transmission. It should be noted that ideal transmission allows each agent to get information with no delay, converging to a centralized scenario. 

The distributed multi-agent diagnosis problem is largely open, and in particular no framework exists to address the inevitable transmission delay impairments associated with information sharing.
Hence, this paper proposes a new diagnosis framework to handle such an issue in DES. 
The framework is able to account for restricted communication among agents, and it bridges the centralized, the distributed and the decentralized architectures in a unified way: in fact, the framework we propose comprises the decentralized scenario as the transmission impairments increase and the centralized one as the impairments vanish.

The main difficulties in developing this framework lie in dealing with the uncertain delays arising from transmitting and processing the information.
Notably, as some information may not contribute to fault diagnosis, a method should be put in place to identify those events whose delays need to be recorded.
Novel methods and structures are put forward, which form the main contributions of this paper:
we propose a new condition for distributed fault diagnosis, namely $ K^T\! $-codiagnosability (cf. Definition \ref{KTCD}), which extends in a natural way the state-of-the-art notion of codiagnosability within $K$ steps, i.e. $ K $-codiagnosability (cf. Definition \ref{KCD});
we propose a novel structure, named as delay recorder (cf. Algorithm \ref{AOS}), to record the delays required for diagnosis;
a new diagnosis function is proposed, with which $ K^T\! $-codiagnosability is verified (cf. Theorem \ref{TKT}).
Here, $T$ refers to the transmission efficiency: our framework comprises the decentralized $K$-codiagnosability as $T$ decreases (i.e. impairments increase), and the centralized $K$-diagnosability as $T$ increases (i.e. the impairments vanish).
Summarizing, this study proposes the first unified framework for distributed fault diagnosis in DES with transmission delay impairments.

The remainder of this paper is organized as follows.
Section \ref{2} describes the partially-observed DES.
Section \ref{3} proposes the key notion of $ K^T\! $-codiagnosability.
In Section \ref{4}, the delay recorder and diagnosis function are proposed to accomplish verification.
Section \ref{8} gives concluding remarks.

\section{Preliminaries}\label{2}

Let us recall the basic formalism of automata, used to model DES. 
Consider the finite event set $E$ in DES as an alphabet, so that the \emph{concatenation} of a string of words in the alphabet can be viewed as a finite sequence of events in $E$. 
A \emph{language} is a set of event strings, formed from the events in $ E $.
Let $E^*$ be the set of all finite strings over $E$. 
Denote the length of a string $ s $ as $|s|$, and let $\epsilon$ be the empty string with $|\epsilon| = 0$. 
%For any language $ L $, the \emph{kleene-closure} of $ L $ is defined as $ L^* = \epsilon \cup L \cup LL \cup \dots $.
The \emph{prefix-closure} of a language $L$ is defined as $\overline{L}=\{s\in E^* | \exists t\in E^*, \text{ s.t. } st\in L\}$, and $L$ is \emph{prefix-closed} if $L=\overline{L}$.
We assume to work with \emph{live} languages $ L $: $\forall s \in L, \exists e\in E $, s.t. $ se\in L$, that is, any string in $ L $ can be extended to arbitrary length.

Automata are a common framework for manipulating languages. 
Let us consider a finite automaton
	\begin{equation}\label{S}
	G=(X, E, \alpha, X_0),
	\end{equation}
	where $X$ is the set of finite states; $E$ is the set of finite events; $ \alpha:X \times E^* \rightarrow 2^X $ is the transition function that describes the transition of an event string; $X_0 \subseteq X $ is the set of possible initial states.
The language generated by $G$ from state $x \in X$ is denoted by $\mathcal{L}(x,G)=\{s \in E^* | \alpha (x,s)! \}$, where $!$ means that the string $ s $ ``is defined'', i.e. it can occur starting from state $x$.
If $x \in X_0$, we simply denote $\alpha(x_0,s)$ as $\alpha(s)$ and $\mathcal{L}(x_0,G)$ as $\mathcal{L}(G)$. 
Given a set of states $ \iota \subseteq X $, we define the set of accessible states of $ \iota $ as $ \mathcal{A}^{G}(\iota) = \{ x'\in X \big| \exists x \in \iota, \exists s \in \mathcal{L}(x,G), \text{ s.t. } x'\in\alpha(x,s) \} $. 

In a partially-observed DES, the event set $ E $ is divided into the observable events $ E_o $ and the unobservable events $ E_{uo} $.
A projection operator $ P_{E_o}: E^* \rightarrow E^*_o $ is used to obtain the observation of an event string:
$\forall s\in \mathcal{L}(G), \forall e\in E: \alpha(se)!, $
\begin{equation}
P_{E_o}(\epsilon)= \epsilon, P_{E_o}(se)=\left\{
\begin{aligned}
& P_{E_o}(s)e, &\text{if} \ e\in E_o; \\
& P_{E_o}(s),  &\text{if} \ e\notin E_o.
\end{aligned}
\right.
\end{equation}
Intuitively, $ P_{E_o}(s) $ shows the observed events for a trajectory $ s\in \mathcal{L}(G) $.
The operator $ P_{E_o} $ can also handle a set of event string, that is, $ \forall S \subseteq \mathcal{L}(G) $, $ P_{E_o}(S)=\{ s\in E_o^*| \exists s' \in S, \text{ s.t. } s=P_{E_o}(s') \} $.
Based on the projection operator $ P_{E_o} $, consider an operation $ \zeta^n_{E_o} $: 
$ \forall s \in \mathcal{L}(G) $,
\begin{equation*}
\zeta^n_{E_o}(s) \!=\! \{ s'' \!\in\! \overline{s} | \exists s' \!\in\! \overline{s}: |s'| \!\geq\! |s|-n, \text{s.t.} P_{E_o}(s'') \!=\! P_{E_o}(s') \}.
\end{equation*}
Intuitively, $ \zeta^n_{E_o}(s) $ collects all prefixes of $ s $ with the same observation as a trajectory $ s' $.

To embed the information of observable events into the system model and determine the indistinguishable states, we make use of the $ M $-machine w.r.t. $ G $ and $ E_o $ as \cite{rudie1995computational,yin2018minimization}
\begin{equation}\label{Mmachine}
\mathcal{M}_{E_o}(G) = (Z, E \cup \{\epsilon\}, \delta, Z_{0}),
\end{equation}
where $ Z \subseteq X \times X $ is the set of states and 
$ Z_0 = X_0 \times X_0 $  is the set of initial states.
For any $ (x_1,x_2) \in Z, e \in E $, the transition function $ \delta: Z \times E \cup \{\epsilon\} \rightarrow 2^Z $ is defined as follows.\\
	$ 1)\ \text{If} \ e\in E_o \wedge \alpha(x_1, e)! \wedge \alpha(x_2, e)!, \text{then}\\
	\hspace*{0.44cm}\delta((x_1,x_2),e) = \{(x_1',x_2')\mid x_1'\in\alpha(x_1,e),x_2'\in\alpha(x_2,e)\}.\\
	2)\ \text{If} \ e\notin E_o \wedge \alpha(x_1, e)!, \text{then}\\
	\hspace*{0.44cm}\delta((x_1,x_2),e) = \{(x_1',x_2)\mid x_1'\in\alpha(x_1,e)\}. \\
	3)\ \text{If} \ e\notin E_o \wedge \alpha(x_2, e)!, \text{then}\\
	\hspace*{0.44cm}\delta((x_1,x_2),\epsilon) = \{(x_1,x_2')\mid x_2'\in\alpha(x_2,e)\}. $\\
Intuitively,  the definitions of $ \alpha $ and $ \delta $ are such that $ \mathcal{L}(G)= \mathcal{L}(\mathcal{M}_{E_o}(G)) $.
% the definitions of $ \alpha $ in $ G $ and $ \delta $ in $ \mathcal{M}_{E_o}(G) $ create a relation between the transitions of $ G $ and $ \mathcal{M}_{E_o}(G) $, that is, $ \mathcal{L}(\mathcal{M}_{E_o}(G))=\mathcal{L}(G) $.
In addition, we have $ \forall s \in \mathcal{L}(G), \delta(s) = \{(x,x') \in Z \mid x \in \alpha(s), x'\in \alpha(s'): P_{E_o}(s')=P_{E_o}(s)\} $, which implies that for any $ (x,x') \in Z $, $ x $ and $ x' $ are indistinguishable with the observation ability $ E_o $. 
In the following, we use the notation $ I_1(x,x') = x $ and $ I_2(x,x') = x' $ to indicate the first and the second state component of $ (x,x') \in Z $.

\subsection{Illustrative example}

To illustrate the key concepts, we present throughout this work a few examples inspired by an air heating unit start-up scenario, cf. Fig. \ref{thermostat}.
At start-up, under healthy conditions, the fan creates an air flow heated by the heating coil. 
The air flow blows the heat away from the coil so that a desired temperature is reached at equilibrium. 
But in some start-up scenarios, the fan may fail to turn on and we need to diagnose the fault to avoid coil overheating. 
The system is monitored by two sensors: a temperature sensor close to the coil, and an air flow sensor at the outlet. 
Denote the event observed by sensor $ 1 $ as $e_1$ (desired temperature is reached) and the event observed by sensor $ 2 $ as $e_2$ (flow rate is regular). 
The fault, obviously unobservable by any sensor, is denoted as $f$. 

\begin{figure}[t]
	\centering
	\subfigure[Air heating unit]
	{\label{thermostat}
		\includegraphics[height=1.28cm]{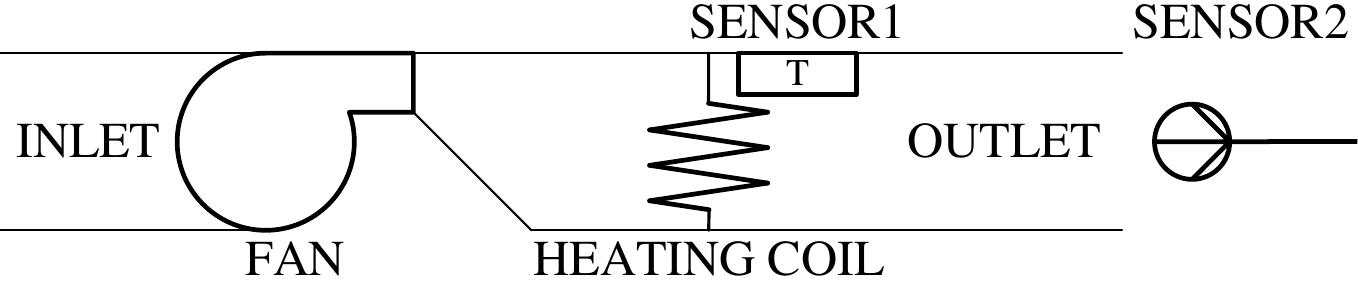}}
	\subfigure[$ G $]
	{\label{System}
		\includegraphics[height=0.9cm]{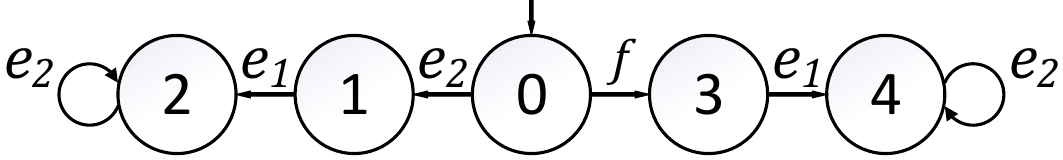}}
	\subfigure[$ \mathcal{M}_{E_o'}(G) $]
	{\label{M}
		\includegraphics[height=2.8cm]{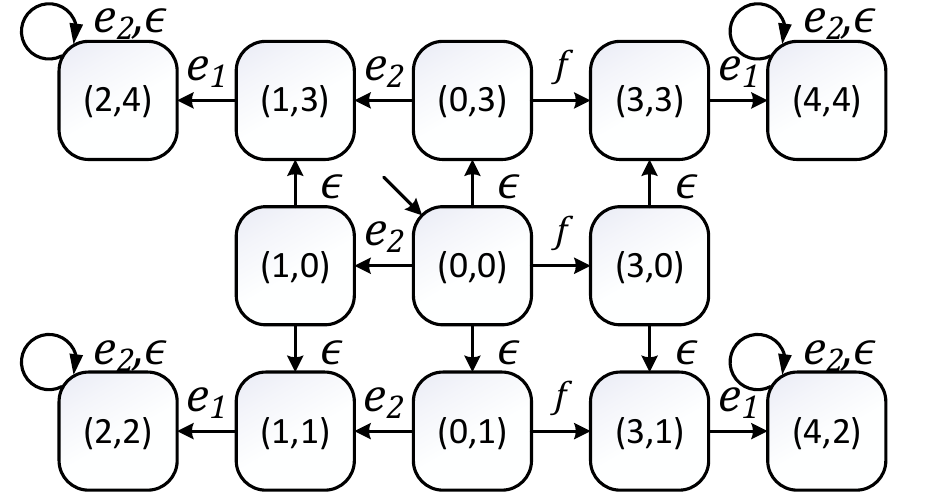}}
	\caption{Air heating unit, the model $ G $ of its start-up process and the $ M $-machine $ \mathcal{M}_{E_o'}(G) $ with $ E_o' = \{e_1\} $}
\end{figure}
	
	\begin{example}\label{ES}(System model).
		We model the start-up of the air heating unit as the automaton $ G $ in Fig. \ref{System}, where the initial state $ x_0 = \{0\} $ means that the system is off initially. 
		The branch on the left of state $ 0 $ represents the healthy functioning: the air flow is regular (in state $ 1 $), so that after some time the desired temperature is reached (in state $ 2 $). 
		The branch starting on the right of state $ 0 $ represents the scenario that the fan does not start, which may be due to an unobservable fault (in state $ 3 $), leading to overheating detected via sensor 1 (in state $ 4 $): however, it is possible that the fan simply did not start timely (e.g. due to blockage in the flow channel, or wear), and after some time, $e_2$ may occur, detected by sensor 2.
		Using the automaton formalism, we have that when the system is in state $ 0 $, the only events that can occur are $ f $ and $ e_2 $, that is, $ \alpha(0,f)! $ and $ \alpha(0,e_2)! $.
		For the string $ fe_1e_2e_2 $ generated by $ G $, let $E_o=\{e_1,e_2\}$.
		Then, $ P_{E_o}(fe_1e_2e_2) = e_1e_2e_2 $, $ \zeta^0_{E_{o}}(fe_1e_2e_2) = \{fe_1e_2e_2\} $ and $ \zeta^3_{E_{o}}(fe_1e_2e_2) = \{\epsilon,f,fe_1,fe_1e_2,fe_1e_2e_2\} $.
		To illustrate the $ M $-machine, consider an observable event set $ E_o' = \{e_1\} $. 
		Then, we have $ \mathcal{M}_{E_o'}(G) = (Z', E \cup \{\epsilon\}, \delta', Z_{0}') $ shown in Fig. \ref{M}.
		For $ fe_1e_2e_2 \in \mathcal{L}(\mathcal{M}_{E_o'}(G)) $, we have $ \delta'(fe_1e_2e_2) = \{(4,4),(4,2)\} $, indicating that states $ 4 $ and $ 2 $ are indistinguishable when we rely only on the observation of $e_1$.
		\hfill \ensuremath{\Box}
	\end{example}

	\section{Distributed Diagnosability}\label{3}
	
	The notion of codiagnosability \cite{debouk2000coordinated} was proposed as the basic property to handle decentralized fault diagnosis, i.e. without information sharing between agents. 
	We provide a distributed extension, called $ K^{T} $-codiagnosability, when information sharing between agents is allowed (possibly subject to transmission impairments).
	Before this, we discuss the distributed observation architecture and ambiguities arising from partial observation and transmission impairments.

	\subsection{Distributed Observation}\label{DO}
	
	Let the system under consideration be monitored by a set of agents $ A = \{a_1, a_2,\dots, a_N\} (N \in \mathbb{N}^+) $ with corresponding events in $ \{E_{1}, E_{2},\dots, E_{N}\} $ such that $ E_o = E_{1} \cup \dots \cup E_{N} $.
	Each agent can share its information with some of the other agents according to a weighted connected graph $ C_A = (V_A, W_A) $ consisting of a set of vertices $ V_A = \{a_1,\dots, a_N\} $ representing the agents, and a set of undirected weighted edges $ W_A \subseteq V_A \times V_A $ representing the transmission links among neighboring agents; the non-negative weight of each edge is related to a transmission delay as specified later.
	Denote the length of the path between two vertices as the sum of the weights along the path.
	%If for any two vertices in $ V_A $ there always exists a path, then $ C_A $ is said to be a connected graph. 
	Then, for any two agents $a_i, a_j$, we define their distance $|a_ia_j|$ as the minimum length between them.

	With the distributed structure above, we now consider a simple communication protocol between agents. 
	Three modules are required for each agent: communication, storage and observation modules, cf. Fig. \ref{agent}, with the following
	\begin{itemize}
		\item \textit{Communication module:} this module forwards $ M_r $ (message received from neighbours) to the storage module, and sends $ M_{new} $ (new message from storage module) and $ M_o $ (message from observation module) to the neighbours. 
		\item \textit{Storage module:} this module stores $ M_o $ from the observation module, and avoids that the occurrence of a certain event is recorded multiple times: in fact, $ M_r $ is stored as $ M_{new} $ only if it is not already in the storage set.
		\item \textit{Observation module:} in this module, a new observed event from sensors receives a timestamp and becomes $ M_o $; the module also performs diagnosis by processing  $ M_o $, $ M_{new} $ with a diagnoser \cite{sampath1995diagnosability} or an observer \cite{cassandras2009introduction}.
		%diagnosis (in the form of, e.g., a diagnoser \cite{sampath1995diagnosability} or an observer \cite{cassandras2009introduction}) uses $ M_o $, $ M_{new} $; the new observed event from sensors receives a timestamp and becomes message $ M_o $. 
	\end{itemize}
	
	\begin{figure}[t]
		\centering
		\includegraphics[height=6.4cm]{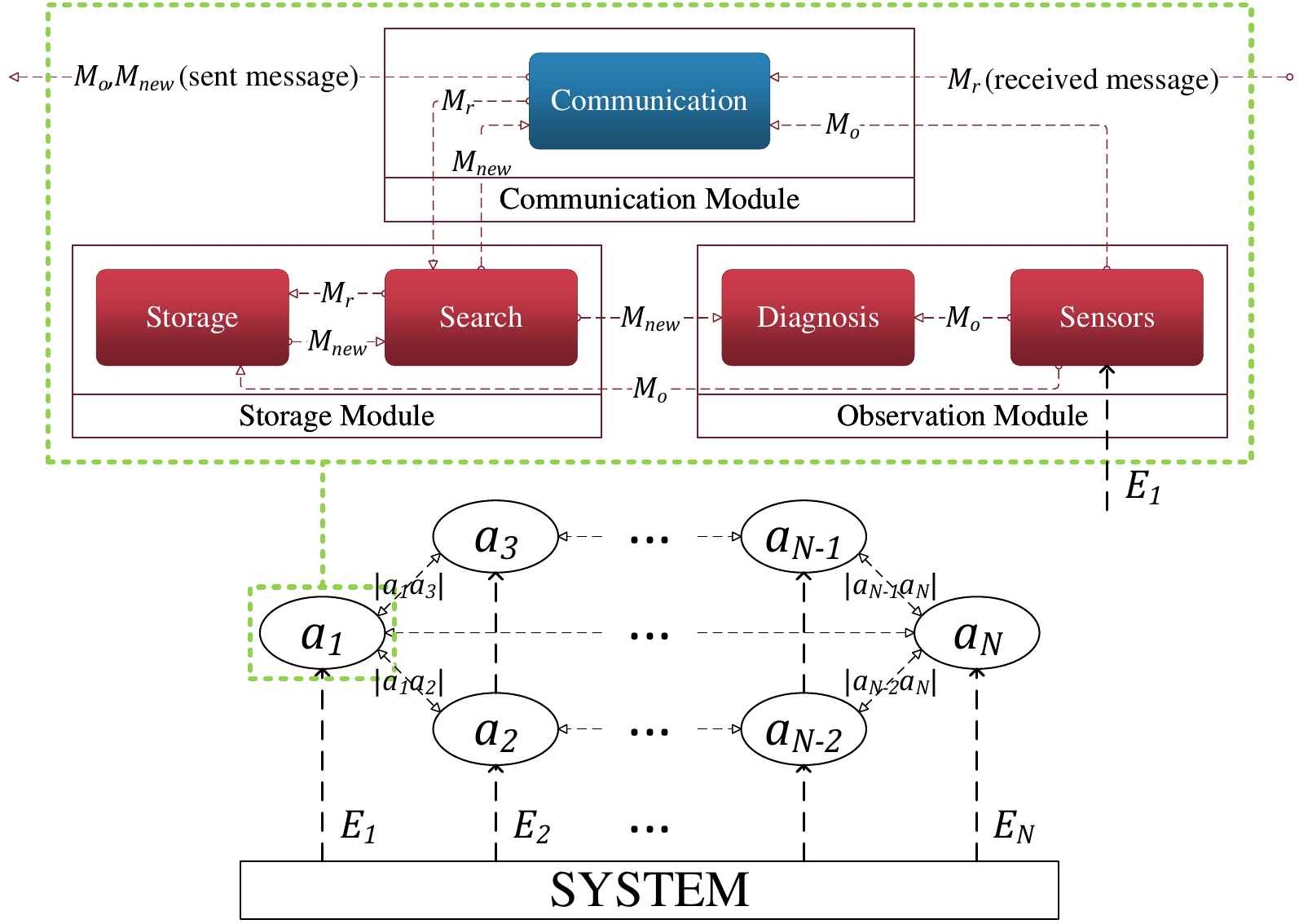}
		\caption{ The distributed observation architecture (lower) and relations between the three modules of each agent (upper) }
		\label{agent}
	\end{figure}
	
	The operations of communication and storage are expected to introduce delays between the observation of an event by one of the agents and the reception of the same event by the other agents.
	Nevertheless, as $ C_A $ is a connected graph, the message that an event is observed (and its timestamp) will reach all other agents, possibly with some delay. 
	For better readability and in line with the literature (e.g.,\cite{debouk2000coordinated,yin2015codiagnosability,viana2019codiagnosabilitya,nunes2018codiagnosability,viana2021codiagnosability}), we will simply analyze two agents $ a_i(i \in \{1,2\}) $.
	All the results in this work can be extended to more agents, at the price that more cases should be analyzed.
	
	Note that the messages processed by each agent $ a_i (i\in\{1,2\})$ have two sources: $ M_o $ from the observation module and $ M_{new} $ from the storage module.
	Thus, the set of all observable events $E_o$ can be partitioned into $E_i$ and $E_o\backslash E_i$, where the occurrence of the events in $E_i$ is received with no delay, while the occurrence of the events in $E_o\backslash E_i$ is received from the storage module with some delay.
	
	%To avoid a trivial scenario, we consider that the time of the occurrence of next events is always unknown.
	We now introduce a coefficient $ T \!>\! 0 $ related to transmission efficiency, where $ T\!=\!1 $ indicates a nominal efficiency and $ T\!<\!1 $ indicates that the efficiency degrades.
	The distance between agents is $ |a_1a_2| \geq 0 $ and we represent the transmission delays as follows: 
	if another agent $a_k$ ($k \!\neq\! i$) observes an event $ e \!\notin\! E_{i} $, $ a_i $ will receive the observation $e$ with a delay of no more than $\lceil \frac{|a_1a_2|}{T} \rceil$ steps ($ \lceil \cdot \rceil $ rounds the element to the nearest integer towards infinity).
	Note that $ T\!\ll\! 1 $ degrades to a decentralized setting, whereas $ T\!\gg\! 1 $ converges to a centralized setting where each agent can monitor \emph{all} observable events with delay of no more than one step.
	
	\subsection{$ K^{T} $-codiagnosability}\label{SKTCD}
	
	%In this subsection, we begin fault analysis in the distributed setting. 
	As a starting point for fault analysis in the distributed	setting, we recall state-of-the-art notions for centralized and decentralized fault diagnosis.
	
	Let $G=(X, E, \alpha, x_0)$ be the system model and $ f $ be the fault events we intend to diagnose. 
	Let $K$ be the maximum number of steps allowed from the occurrence of a fault to its diagnosis.
	To diagnose the faults within $K$ steps, a structure of step counter $ \Delta: \mathcal{L}(G) \to \{-1, 0, 1, \dots, K\} $ is used to count the number of steps in an event string after a fault occurs: $ \forall s \in \mathcal{L}(G), \forall e\in E: \alpha(se)! \Rightarrow \Delta(\epsilon)=-1, \Delta(se)= $
	\begin{align}\label{Delta}
	\left\{
	\begin{aligned}
	& \Delta(s), &\text{if}\ & [\Delta(s) \!=\! -1 \wedge e \!\neq\! f ]\vee [\Delta(s) \!=\! K]; \\
	& \Delta(s) \!+\! 1, &\text{if}\ &  [\Delta(s) \!=\! -1 \wedge e \!=\! f ]\vee [0 \!\leq\! \Delta(s) \!<\! K];
	\end{aligned}
	\right.
	\end{align}
	where $ -1 $ means no fault happens.
	By means of $ \Delta $, the literature has introduced the notions of $K$-diagnosability and $K$-codiagnosability:
	\begin{definition}\label{KD}
		($ K $-diagnosability\cite{cassez2008fault})  
		For $ K \in \mathbb{N} $, the live language $ \mathcal{L}(G) $ is $ K $-diagnosable w.r.t. $ f $ if $ \forall s \in \mathcal{L}(G): \Delta(s) = K $,
		\begin{equation}\label{KDC}
		\forall s'\in \mathcal{L}(G): P_{E_o}(s')=P_{E_o}(s) , \Delta(s') \neq -1.
		\end{equation}
	\end{definition}
\vspace{0.6em}
	\begin{definition}\label{KCD}
		($ K $-codiagnosability\cite{cassez2012complexity})  
		For $ K \in \mathbb{N} $, the live language $ \mathcal{L}(G) $ is $ K $-codiagnosable w.r.t. $ f $ if $ \forall s \in \mathcal{L}(G): \Delta(s) = K, $ 
		\begin{equation}\label{KCDC}
		\exists i \!\in\! \{1,2\}, \text{s.t.} \forall s'\!\in\! \mathcal{L}(G)\!:\! P_{E_i}(s')\!=\!P_{E_i}(s) , \Delta(s') \!\neq\! -1. \ \!\!\!\!\!\!
		\end{equation}
	\end{definition}
	\vspace{0.6em}
	
	Obviously, $K$-diagnosability is a centralized notion as a single monitors all observable events.
	In $ K $-codiagnosability, $ f $ can be diagnosed by either $ a_1 $ or $ a_2 $ unambiguously within $ K $ steps, without any communication between agents. 
	Unfortunately, the following example shows that some faults may go undetected in the absence of communication.
	
	\begin{example}\label{ELO}(Limits of $ K $-codiagnosability).
		For the system in Fig. \ref{System}, we consider $ a_1 $ with observation ability $ E_{1} = \{e_1\} $, and $ a_2 $ with observation ability $ E_{2} = \{e_2\} $. 
		Since $ P_{E_{1}}(fe_1e_2e_2)=P_{E_{1}}(e_2e_1e_2)=e_1 $ and $ P_{E_{2}}(fe_1e_2e_2)=P_{E_{2}}(e_2e_1e_2)=e_2e_2 $, i.e., faulty and healthy strings are indistinguishable, it is impossible for $ a_1 $ or $ a_2 $ to determine within $K=3$ steps if the fault $ f $ has occurred or not. 
		We conclude that, when $K=3$, $ \mathcal{L}(G) $ is not $ K $-codiagnosable w.r.t. $ f $.
		\hfill \ensuremath{\Box}
	\end{example}
	
	Intuitively, a fault that goes undetected in the absence of communication may become detectable if communication among agents is allowed (cf. Example \ref{EKTCD}). 
	This means that $ K $-codiagnosability is restrictive and an appropriate extension is required, which is the key definition in this paper:
	\begin{definition}\label{KTCD}
		($ K^{T} $-codiagnosability) For $ K \in \mathbb{N} $ and $ T > 0 $, the live language $ \mathcal{L}(G) $ is $ K^{T} $-codiagnosable w.r.t. $ f $ if $ \forall s \in \mathcal{L}(G): \Delta(s) = K, $
		\begin{align}\label{KTCDC}
		& \exists i \in \{1,2\}, \text{ s.t. } \forall s'\in \mathcal{L}(G): P_{E_i}(s') = P_{E_i}(s)\ \wedge    \nonumber \\
		& P_{E_o}(\zeta_{E_o \backslash E_i}^{\lceil \frac{|a_1a_2|}{T} \rceil}(s')) \cap P_{E_o}(\zeta_{E_o \backslash E_i}^{\lceil \frac{|a_1a_2|}{T} \rceil}(s)) \!\neq\! \emptyset, \Delta(s') \!\neq\! -1.
		\end{align}
	\end{definition}
	\vspace{0.6em}
	
	Intuitively, if (\ref{KTCDC}) holds, then $ a_i $ can timely observe or receive all key events to determine the occurrence of $ f $.
	In other words, $ a_i $ is capable of consistently distinguishing a fault string ($ s $ satisfying $ \Delta(s) = K $) from a normal string ($ s $ satisfying $ \Delta(s) = -1 $), despite the imperfect observation caused by delay. 
	As $ T $ increases, the strings that cannot be distinguished from the string $ s: \Delta(s) = K $ become less and less, that is, (\ref{KTCDC}) gets weaker and weaker.
	As expected, $ K^{T\!} $-codiagnosability $ \Rightarrow $ $ K^{T'}\! $-codiagnosability when $ T \leq T' $ (higher transmission efficiency improves diagnosis ability).
	
	\begin{example}\label{EKTCD}($ K^{T} $-codiagnosability).
		Consider the same system and agents as Example \ref{ELO}.
		Suppose $ |a_1a_2|=2 $, $ T'=2 $, then $ a_1 $ will receive the occurrence of $ e_2 $ in no more than $ \lceil \frac{|a_1a_2|}{T'} \rceil=1 $ step.
		When $ e_2e_1 $ occurs, the occurrence of $ e_2 $ will be received by $a_1$ before $ e_1 $, but no $ e_2 $ will be received by $a_1$ before the observation of $ e_1 $ when $ fe_1 $ occurs.
		This implies that $ e_2e_1 $ and $ fe_1 $ are distinguishable, i.e., the fault can be diagnosed by $ a_1 $.
		Indeed, Definition \ref{KTCD} gives $ P_{E_1}(fe_1)=e_1 $, $ P_{E_o}(\zeta_{E_2}^{1}(fe_1))=\{\epsilon\} $, and each string $ s' \in \{s|P_{E_1}(s)=e_1 \wedge P_{E_o}(\zeta_{E_2}^{1}(s))\cap\{\epsilon\} \neq \emptyset\} = \{fe_1,fe_1e_2\} $ satisfies $ \Delta(s') \geq 0 $.
		We conclude that $ \mathcal{L}(G) $ is $ K^{T'} $-codiagnosable w.r.t. $ f $ when $ T'=2 $ and $ K\geq1 $.
		Next, suppose $ T=1 $, so that $ a_1 $ will receive the occurrence of $ e_2 $ in no more than $ \lceil \frac{|a_1a_2|}{T} \rceil=2 $ steps.
		In this case, we know that $ fe_1e_2e_2 $ and $ e_2e_1 $ are indistinguishable since $ e_2 $ may not be received before $ e_1 $.
		Correspondingly, the string set $ \{s\mid P_{E_1}(s)=P_{E_1}(fe_1e_2e_2) \wedge P_{E_o}(\zeta_{E_2}^{2}(s))\cap P_{E_o}(\zeta_{E_2}^{2}(fe_1e_2e_2)) \neq \emptyset\} = \{e_2e_1,fe_1e_2,fe_1e_2e_2,\dots\} $ and $ \Delta(e_2e_1)=-1 $.
		That is, the fault may not be diagnosed by $ a_1 $ when $ T=1 $ and $ K = 3 $.
		Nevertheless, Example \ref{EKT} will show that $ K^{T} $-codiagnosability is satisfied when $ T=1 $ and $ K = 3 $, as the fault can be diagnosed by $ a_2 $.
		\hfill \ensuremath{\Box}
	\end{example}

	\begin{remark}\label{rr} 
		(Relations between $ K $-codiagnosability, $ K^{T} $-codiagnosability and $ K $-diagnosability).
		Obviously, (\ref{KCDC}) $ \Rightarrow $ (\ref{KTCDC}), that is, if $ K $-codiagnosability holds, then $ K^T\! $-codiagnosability holds for any $ T $.
		From (\ref{KTCDC}) and the definition of $ \zeta $, we have $ P_{E_o}(s') \!=\! P_{E_o}(s) \Leftrightarrow P_{E_o}(\zeta_{{E_o \backslash E_i}}^{0}(s')) \!=\! P_{E_o}(\zeta_{{E_o \backslash E_i}}^{0}(s)) \Rightarrow P_{E_i}(s') \!=\! P_{E_i}(s) \wedge P_{E_o}(\zeta_{E_o \backslash E_i}^{\lceil \frac{|a_1a_2|}{T} \rceil}(s')) \!\cap\! P_{E_o}(\zeta_{E_o \backslash E_i}^{\lceil \frac{|a_1a_2|}{T} \rceil}(s)) \!\neq\! \emptyset $ for any $ T $.
		We obtain that (\ref{KTCDC}) $ \Rightarrow $ (\ref{KDC}), that is, if $ K^T\! $-codiagnosability holds for any $ T $, then $ K $-diagnosability holds.
		Hence, we conclude:\\
		$ K\!\text{-codiagnosability} \;\!\!\!\Rightarrow\;\!\!\! K^T\;\;\!\!\!\!\!\text{-codiagnosability} \;\!\!\!\Rightarrow\;\!\!\! K\!\text{-diagnosability}. $
	\end{remark}

\section{Verification of Distributed Diagnosability}\label{4}

In general, it is impossible to determine if a fault can be diagnosed by analysing each event string as in Example \ref{EKTCD}.
It is necessary to embed the delay information into the automaton and develop a feasible method to verify $ K^T\! $-codiagnosability.
This is done by linking the system states to fault events and by building a delay recorder structure to handle the delays.

\subsection{Delay Recorder}\label{SDR}

In diagnosis, it is crucial to observe the events that help to distinguish the faulty from the healthy strings.
A delay recorder aims to register the delays of these events correctly.
%In fact, although any information received from the other agents comes with a delay, 
Note that recording all the delays can be deleterious for verification, which will be more clear in Example \ref{EI}.

Motivated by the step counter $ \Delta $ in (\ref{Delta}), a structure of fault automaton is constructed from \eqref{S} to count the number of steps after a fault happens \cite{yin2015codiagnosability}
\begin{equation}\label{FA}
\hat{G}=(\hat{X}, E, \hat{\alpha}, \hat{X}_{0}),
\end{equation}
where $ \hat{X} = X\times \{-1, 0, 1, \dots, K\} $ includes the state in $ X $ and the fault counting component, i.e. $ \hat{x} = (x,|\hat{x}|_f) \in \hat{X} $ where $ |\hat{x}|_f $ indicates the number of steps after a fault occurs, as calculated in \cite{yin2015codiagnosability}. 
The transition function $\hat{\alpha}: \hat{X} \times E \to 2^{\hat{X}} $ is defined as: 
for any $ \hat{x}=(x,|\hat{x}|_f)\in \hat{X} $ and $ e\in E $ satisfying $ \alpha(x,e)! $, we have $\hat{\alpha}((x,|\hat{x}|_f),e) \!=\! \{(x',|\hat{x}|_f+v) | x' \!\in\! \alpha(x,e)\},$ where $ |\hat{x}|_f \in \{-1, 0, 1, \dots, K\} $ and $ v $ is defined by

$ \
v = \left\{
\begin{aligned}
& 0 ,\quad \text{if} \ [|\hat{x}|_f = -1 \wedge e \neq f ]\vee [|\hat{x}|_f = K];\\
& 1 ,\quad \text{if} \ [|\hat{x}|_f = -1 \wedge e = f ]\vee [0 \leq |\hat{x}|_f < K].
\end{aligned}
\right.
$\\
The set of initial states is $ \hat{X}_{0} =\{(x_0, -1)\mid x_0\in X_0 \}$.
Obviously, $ \mathcal{L}(\hat{G}) = \mathcal{L}(G) $.
	Let us consider the $ M $-machine w.r.t. $ \hat{G} $ and $ E_o $, that is, $ \mathcal{M}_{E_o}(\hat{G})=(Z, E \cup \{\epsilon\}, \delta, Z_{0}) $. 
	For each state $ z \in Z \subseteq \hat{X} \times \hat{X} $, we have that $ I_1(z),I_2(z)\in \hat{X} $: thus, we can use $ |I_1(z)|_f $ and $ |I_2(z)|_f $ to represent the fault counting value.
	We denote the ``confusing state'' subset as
	\begin{equation}\label{key}
	Z^C=\{z:|I_1(z)|_f = K \wedge |I_2(z)|_f = -1\} \subseteq Z.
	\end{equation}

We are now in the position to explain how to handle delays.
For agent $ a_i (i\in\{1,2\})$, the delay recorder to determine the delays to be recorded is defined as an automaton
\begin{equation}\label{DR}
\mathcal{R}_i = (X_i, E, \alpha_i, X_{i0}),
\end{equation}
where each state $ x_i=(\hat{x}, (|x_i|^{j_i}_i)^{j_i}) \in X_i \subseteq \hat{X} \times (\{0, 1, \dots, \lceil \frac{|a_1a_2|}{T} \rceil, \infty\})^{j_i} (j_i\in\{1,\dots,n_i\}) $ contains 2 components: the state in $ \hat{X} $, and the delay value, denoted by $ |x_i|_i^{j_i} $.
Here, $ n_i $ is the number of the delay value we need to record.
The transition function $ \alpha_i: X_i \times E \to 2^{X_i} $ and the set of initial states $ X_{i0} $ are built as in Algorithm \ref{AOS}.

\makeatletter
\def\BState{\State\hskip-\ALG@thistlm}
\makeatother

\begin{algorithm} [htbp]
	\begin{spacing}{0.95}
	\begin{algorithmic}[1]
		\caption{The construction of the delay recorder $ \mathcal{R}_i $} \label{AOS}
		\Require
		$ \hat{G} = (\hat{X}, E, \hat{\alpha}, \hat{X}_{0}), \lceil \frac{|a_1a_2|}{T} \rceil, Z^C_{E_i}, E_o, E_i (i \in \{1,2\}) $;
		\Ensure
		$ \mathcal{R}_i = (X_i, E, \alpha_i, X_{i0}) $;
		
		\State $ T^{ini}, T_i^1, X_i^1, Y_i^1, X_{i0}, X_{i} \gets \emptyset $; $ m \gets 1 $;
		\State \texttt{FORWARD1}$ (\hat{X}_{0}, X_i^1, Y_i^1) $; \texttt{FORWARD2}$ (Y_i^1) $; $ n \gets m $;
		\For {$ l \in \{1,\dots,n\} $}
		\State $ I \!\gets\! \bigcup_{y\in Y_i^{l\;\!\!\;\!\!-\;\!\!\;\!\!n\;\!\!\;\!\!+\;\!\!\;\!\!m}} \{\{\hat{x}' \!\in\! \hat{X} | \exists \hat{x} \!\in\! \hat{X} \!:\! (\hat{\alpha}(\hat{x},e)\!=\!\hat{x}') \!\in\! y\}\} $;
		\If {$ \forall (\hat{x}^k,\hat{x}^{-1})\!\in\! Z_{E_i}^C, [\forall \iota,\iota' \!\in\! I \!:\! \iota \!\neq\! \iota', \hat{x}^k \!\notin\! \mathcal{A}^{\hat{G}}(\iota), $ \hspace*{0.42cm} $ \hat{x}^{-1} \!\notin\! \mathcal{A}^{\hat{G}}(\iota')] \wedge [\hat{x}^k \!\notin\! X_i^{l-n+m} \vee \hat{x}^{-1} \!\notin\! X_i^{l-n+m}]  $}
		\State $ T_i^{l-n+m} \!\gets\! T_i^{m} $; $ X_i^{l-n+m} \!\gets\! X_i^{m} $; $ Y_i^{l-n+m} \!\gets\! Y_i^{m} $;
		\State $ m \!\gets\! m-1 $;
		\EndIf
		\EndFor
		\State $ n \gets m $; $ n' \gets 1 $; $ n_i \gets 1 $;
		\For {$ j \in \{2,3,\dots,n\} $}
		\If {$ \forall l \in \{1,\dots,n_i\}, \exists y \in Y_i^j, $s.t.$ \forall y' \in Y_i^{l}, y \not\subseteq y' $}
		\For {$ l' \in \{0,\dots,n'\!-1\} $}
		\If {$ \forall y \in Y_i^{n'\!-l'} $, $ \exists y' \in Y_i^{j} $, s.t. $ y \subseteq y' $}
		\State $ T_{i}^{n'\!-l'} \!\!\gets\! T_{i}^{n_i} $; $ Y_{i}^{n'\!-l'} \!\!\gets\! Y_{i}^{n_i} $; $ n_i \!\gets\! n_i\!-\!1 $; 
		\EndIf
		\EndFor
		\State $ n_i \gets n_i+1 $; $ n' \gets n_i $; $ T_i^{n_i} \gets T_i^{j} $; $ Y_i^{n_i} \gets Y_i^{j} $;
		\EndIf
		\EndFor
		\For {$ \hat{x}_0\in\hat{X}_0 $}
		\State $ X_{i0} \gets X_{i0} \cup \{(\hat{x}_{0},(\infty)^{1})\} $; $ X_{i} \gets X_{i} \cup \{(\hat{x}_{0},(\infty)^{1})\} $;
		\State \texttt{RECORD}($ (\hat{x}_{0},(\infty)^{1}) $,$ \mathcal{R}_i $); 
		\For {$ j \in \{2,\dots,n_i\} $}
		\State $ X_{i0} \!\gets\! X_{i0} \!\cup\! \{(\hat{x}_{0},(0)^{j})\} $; $ X_{i} \!\gets\! X_{i} \!\cup\! \{(\hat{x}_{0},(0)^{j})\} $;
		\State \texttt{RECORD}($ (\hat{x}_{0},(0)^j) $,$ \mathcal{R}_i $); 
		\EndFor
		\EndFor
		\State \textbf{Return} $ \mathcal{R}_i $;
		\vspace{0.19em}
		\Procedure{\texttt{FORWARD1}}{$ \iota, X_i^j, Y_i^j $}
		\State $ X^{tem} \!\gets\! \{\hat{x} \!\in\! \hat{X} | \exists \hat{x}' \!\in\! \iota, s \!\in\! (E\backslash E_o)^*,\text{s.t.} \!\ \hat{x} \!\in\! \hat{\alpha}(\hat{x}',s)\} $; 
		\If {$ X^{tem} \not\subseteq X_i^j $}
		\State $ X_i^j \gets X_i^j \cup X^{tem} $;
		\For {$ e \in E_o: \exists \hat{x} \in X^{tem}, \text{ s.t. } \hat{\alpha}(\hat{x},e)! $}
		\If {$ e \in E_i $}
		\State \texttt{FORWARD1}$ (\{\hat{x} \in X_i | \exists \hat{x}' \in X^{tem}, $ s.t. $ \hat{x} \in \hspace*{2.01cm} \hat{\alpha}(\hat{x}',e)\}, X_{i}^{j}, Y_{i}^{j}) $; 
		\ElsIf {$ e \in E_o\backslash E_i $}
		\State $ Y_{i}^{j} \gets Y_{i}^{j} \cup \{\{\hat{\alpha}(\hat{x}',e)=\hat{x} | \exists \hat{x}' \in X^{tem}, $ \hspace*{2.01cm} $ \hat{x} \in \hat{X}, \text{ s.t. } \hat{x} \in  \hat{\alpha}(\hat{x}',e)\}\} $;
		\EndIf 
		\EndFor
		\EndIf
		\EndProcedure
		\Procedure{\texttt{FORWARD2}}{$ Y_i^j $}
		\For {$ y \in Y_i^{j} $}
		\State $ X' \gets \{\hat{x}' \in \hat{X} | \exists \hat{x} \in \hat{X}: (\hat{\alpha}(\hat{x},e)=\hat{x}') \in y\} $;
		\If {$ [\exists (\hat{x}^k,\hat{x}^{-1})\in Z_{E_i}^C, \text{ s.t. } \hat{x}^k,\hat{x}^{-1} \in \mathcal{A}^{\hat{G}}(X')] \wedge \hspace*{0.96cm} [y\notin T^{ini}] $}
		\State $ T^{ini} \gets T^{ini} \cup y $; $ m \gets m+1 $; 
		\State $ T_{i}^{m} \gets y $; $ X_{i}^{m} \gets \emptyset $; $ Y_{i}^{m} \gets \emptyset $;
		\State \texttt{FORWARD1}$ (X', X_i^{m}, Y_i^{m}) $; \texttt{FORWARD2}$ (Y_i^{m}) $;
		\EndIf
		\EndFor
		\EndProcedure
	\end{algorithmic}
	\end{spacing}
\end{algorithm}

	\begin{algorithm} [htbp]
		\begin{spacing}{0.95}
		\begin{algorithmic}[1]
			\caption{The \texttt{RECORD} procedure in Algorithm \ref{AOS}} \label{AP}
			\Procedure{\texttt{RECORD}}{$ (\hat{x},(u)^j),\mathcal{R}_i $}
			\For {$ e \in E: \hat{\alpha}(\hat{x},e)! $}
			\For {$ \hat{x}'\in \hat{\alpha}(\hat{x},e) $}
			\If {$ (\hat{\alpha}(\hat{x},e)=\hat{x}') \in T_i^j $}
			\State $ u' \gets \infty $;
			\ElsIf {$ u=\infty \wedge e \in E_o\backslash E_i $}
			\State $ u' \gets \lceil \frac{|a_1a_2|}{T} \rceil $;
			\ElsIf {$ u>0 $}
			\State $ u' \gets u-1 $;
			\ElsIf {$ u=0 $}
			\State $ u' \gets 0 $;
			\EndIf
			\State Add $ \alpha_i((\hat{x},(u)^j),e) = (\hat{x}',(u')^j) $ to $ \mathcal{R}_i $;
			\If {$ (\hat{x}',(u')^j) \notin X_i $}
			\State $ X_i \gets X_i \cup \{(\hat{x}',(u')^j)\} $;
			\State \texttt{RECORD}($ (\hat{x}',(u')^j), \mathcal{R}_i $);
			\EndIf
			\EndFor
			\EndFor
			\EndProcedure
		\end{algorithmic}
		\end{spacing}
	\end{algorithm}

	Algorithm \ref{AOS} consists of two parts: the first part (lines 1-20) marks $ n_i $ sets of transitions, utilized to build the delay recorders with the \texttt{RECORD} procedure in the second part (lines 21-29).
	All the procedures are listed in Algorithm \ref{AP}.
	The procedure \texttt{FORWARD1}($ \iota, Y_i^j $) collects the transition sets for $ Y_i^j $, that, only containing the events in $ E_o \backslash E_i $, cannot be distinguished under the observation ability $ E_o $.
	The procedure \texttt{FORWARD2}($ Y_i^j $) explores new transition sets; lines 3-20 check which transition set in $ Y_i^j $ is necessary to be marked.
	The procedure \texttt{RECORD} 
	%builds the delay recorder with each set of transitions marked in the first part of Algorithm \ref{AOS}.
	%To be specific, we 
	appends the delay $ \infty $ to the subsequent states of the marked transitions, and records the delay of the transitions that ``leave'' these states marked by $ \infty $.
	
	As $\mathcal{R}_i$ is built from $ \hat{G} $, we have $ \mathcal{L}(\mathcal{R}_i) = \mathcal{L}(\hat{G}) $.
	Note that $\mathcal{R}_i$ records the delay of the events in $ E_o\backslash E_i $ that help to distinguish the pair of system trajectories $ s $ and $ s' $ satisfying $ P_{E_i}(s)=P_{E_i}(s') $ and $ \Delta(s)=K, \Delta(s')=-1 $. 
	This is needed to verify $ K^T\! $-codiagnosability, as it will be clear in the next section.
	
	\begin{example}\label{EI}(The importance of a delay recorder).
		For the system $ G $ in Fig. \ref{System}, the corresponding fault automaton $ \hat{G} $ is shown in Fig. \ref{FCA}.
		With the observable event set $ E_1 = \{e_1\} $, we obtain the $ M $-machine $ \mathcal{M}_{E_1}(\hat{G}) $ and the confusing state subset $ Z_{E_1}^C = \{((4,3),(2,-1))\} $.
		Next, using $ \hat{G} $, $ |a_1a_2|=2 $, $ T'=2 $, $ Z_{E_1}^C $,$ E_o $ and $ E_1 $, we run Algorithm \ref{AOS} to obtain the delay recorder $ \mathcal{R}_1' $ shown in Fig. \ref{R1'}.
		From $ \mathcal{R}_1' $, one can see that only the delay of $ \mathbf{e_2} $ (denoted with bold) in $ fe_1\mathbf{e_2}e_2 $ and $ \mathbf{e_2}e_1e_2 $ are recorded.
		The delay value ``$ 0 $" in $ ((4,3),(0)^1) $ and $ ((2,-1),(0)^1) $ indicates that the occurrence of these $ \mathbf{e_2} $ that help to distinguish $ fe_1e_2e_2 $ and $ e_2e_1e_2 $ have been received, which implies that the fault can be diagnosed by $ a_1 $, as shown in Example \ref{EKTCD} with $ T'=2 $.
		Nevertheless, without a delay recorder, a naive strategy could be to record each delay of $\mathbf{e_2}$, that is, $ fe_1\mathbf{e_2}\mathbf{e_2} $ and $ \mathbf{e_2}e_1\mathbf{e_2} $.
		Unfortunately, by doing this, one would obtain $ ((4,3),(1)^1) $ and $ ((2,-1),(1)^1) $, where the delay value ``$ 1 $" implies that we cannot determine if $ e_2 $ has been received by $ a_1 $, leading to a trouble for verification.
		\hfill \ensuremath{\Box}
	\end{example}
	
	\begin{figure}[htbp]
		\centering
		\subfigure[$ \hat{G} $]
		{\label{FCA}
			\includegraphics[height=0.9cm]{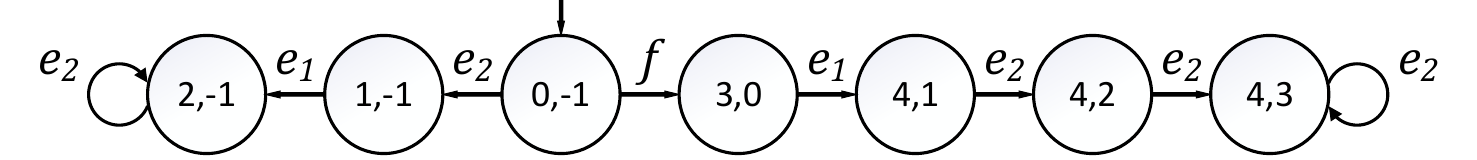}}
		\subfigure[$ \mathcal{R}_{1}' $]
		{\label{R1'}
			\includegraphics[height=0.9cm]{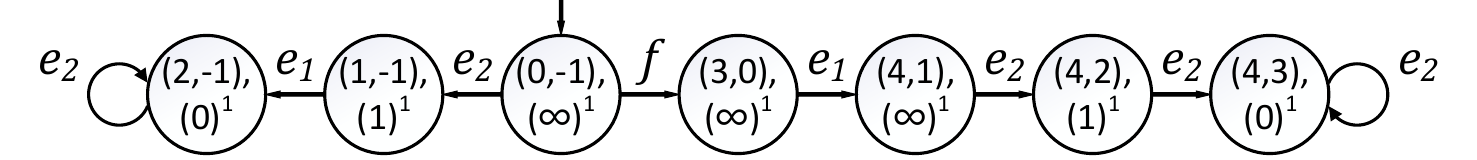}}
		\caption{Fault automaton $ \hat{G} $ and delay recorder $ \mathcal{R}_{1}' $}
	\end{figure}
	
	\begin{remark} (Complexity analysis).
		Like existing diagnosis algorithms, Algorithm \ref{AOS} relies on the construction of observers (or diagnosers).
		This operation is known from the literature having worst-case exponential complexity $ O(2^{2n}) $ \cite{cassandras2009introduction}, where $ n $ is the number of states. 
		However, this worst-case complexity is rarely reached, and recent studies have shown that, for deterministic automata, the average state size of diagnosers is $ O(n^{0.77\log k+0.63}) $ \cite{clavijo2017empirical}, where $ k $ is the number of events.
	\end{remark}
	
	\subsection{Verification of $ K^T\! $-codiagnosability}\label{SV}
	
	Using all structures introduced before, we now initiate the verification process of $ K^T\! $-codiagnosability.
	
	Let $ \hat{G} $ in (\ref{FA}) be the fault automaton built from $ G $ in (\ref{S}).
	We first consider the delay value of $ a_i $ ($ i\in\{1,2\} $).
	To run Algorithm \ref{AOS}, we built $ \mathcal{M}_{E_{i}}(\hat{G}) $ to get the state subset $ Z^C_{E_i} $, and then the delay recorder $ \mathcal{R}_i $ in (\ref{DR}) is obtained.
	Aiming to determine the fault states that $ a_i $ cannot diagnose even with the received message, we further build the $ M $-machine $ \mathcal{M}_{E_{i}}(\mathcal{R}_i) =(Z_i, E \cup \{\epsilon\}, \delta_i, Z_{i0}) $.
	As $ \mathcal{M}_{E_{i}}(\mathcal{R}_i) $ only contains the delay information of $ a_i $, we need to run Algorithm \ref{AOS} again with $ E_k $ ($ k\in\{1,2\}, k\neq i $) to obtain the delay information of $ a_k $.
	Considering that the input of Algorithm \ref{AOS} should be a fault automaton, we need reconstruct $ \mathcal{M}_{E_{i}}(\mathcal{R}_i) $ to be a fault automaton with the delay information of $ a_i $.
	To this end, we remove the second component of each state in $ Z_i $ and the empty event $ \epsilon $ in event set $ E $, constructing the automaton 
	\begin{equation}\label{hat}
	\hat{\mathcal{R}}_i=(\hat{X}_i, E, \hat{\alpha}_i, \hat{X}_{i0}).
	\end{equation}
	The state $ \hat{X}_i = \hat{X} \times (\{0, 1, \dots, \lceil \frac{|a_1a_2|}{T} \rceil, \infty\})^{j_i} \times \{H,F\} $ ($ j_i \in \{1,\dots,n_i\} $), where ``$ H $'' means ``healthy'', ``$ F $'' means ``faulty''.
	For each $ z_i \in Z_i $ and the corresponding $ \hat{x}_i = ((x,|\hat{x}_i|_f),(|\hat{x}_i|_i^{j_i})^{j_i},|\hat{x}_i|_d) \in \hat{X}_i $, denote $ |\hat{x}_i|_f = |I_1(z_i)|_f $, $ |\hat{x}_i|_i^{j_i} = |I_1(z_i)|_i^{j_i} $ and
	\begin{equation}\label{D}
	|\hat{x}_i|_d = \left\{
	\begin{aligned}
	F , & \quad \text{if } z_i \in \mathcal{VC}_i;\\
	H , & \quad \text{otherwise};
	\end{aligned}
	\right.
	\end{equation}
	where we define the condition $ z_i \in \mathcal{VC}_i $ as:
	\begin{equation}\label{vci}
	\!\!\!\!\!\!\!\!\!\!\!\!\!\!\!\!\!\!\! |I_1(z_i)|_f \!=\! K, |I_2(z_i)|_f \!=\! -1, |I_1(z_i)|_i^{j_i} \!>\! 0, |I_2(z_i)|_i^{j_i} \!>\! 0. \!\!\!\!\!\!\!\!\!\!
	\end{equation}
	Clearly, $ \hat{\mathcal{R}}_i $ can be seen as a fault automaton: 
	in fact, each state $ \hat{x}_i \in \hat{X}_i $ has a fault counting value, and $ |\hat{x}_i|_d = F $ can be regarded as the fault states $ \hat{x} \in \hat{X} $ satisfying $ |\hat{x}|_f=K $ in the automaton $ \hat{G} $ to determine the confusing state subset $ Z_{E_k}^C $.

	\begin{example}\label{EM}(The  reconstructed automaton).
		For the fault automaton $ \hat{G} $ and the state subset $ Z_{E_1}^C $ in Example \ref{EI}, we run Algorithm \ref{AOS} with $ T \!=\! 1 $ to obtain $ \mathcal{R}_1 $ shown in Fig. \ref{R2}.
		The crucial difference between $ \mathcal{R}_1 $ and $ \hat{\mathcal{R}}_1 $ (shown in Fig. \ref{DR2}) is the third component $ \{H,F\} $.
		For compactness, Fig. \ref{RM2} shows the $ M $-machine $ \mathcal{M}_{E_{1}}(\mathcal{R}_1) \!=\! (Z_1, E \cup \{\epsilon\}, \delta_1, Z_{10}) $ after omitting the states $ z_1 \in Z_1 $ satisfying $ \mathcal{A}^{\mathcal{R}_1}(I_1(z_1)) \cap \{z_1' \in Z_1 | |I_1(z_1')|_f\!=\!K\} \!=\! \emptyset\ \vee $ $ |I_2(z_i)|_f\!>\!0 $
		(according to (\ref{vci}), all states $ \hat{x}_1 \in \hat{X}_1 $ corresponding to the omitted states in Fig. \ref{RM2} must satisfy $ |\hat{x}_1|_d\!=\!H $). 
		For the string $ fe_1e_2e_2 \in \mathcal{L}(\mathcal{M}_{E_{1}}(\mathcal{R}_1)) $, we have $ \delta_1(fe_1e_2e_2) \!=\! \{(((4,3),(1)^1),((2,-1),(0)^1)),$ $(((4,3),(1)^1),((2,-1),(1)^1)),\dots\} $, which corresponds to $ \hat{\alpha}_1(fe_1e_2e_2)\!=\!\{((4,3),(1)^1,H),((4,3),(1)^1,F)\} $.
		The state $ ((4,3),(1)^1,F) $ corresponds to the fact, shown in Example \ref{EKTCD}, that the fault may not be diagnosed by $ a_1 $.
		\hfill \ensuremath{\Box}
	\end{example}
	
	\begin{figure}[t]
		\centering
		\subfigure[$ \mathcal{R}_{1} $]
		{\label{R2}
			\includegraphics[height=1.7cm]{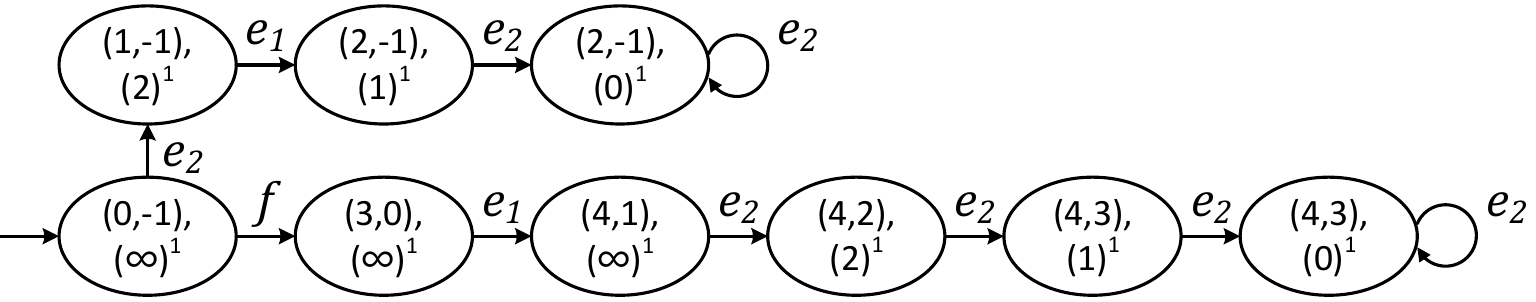}}
		\subfigure[Part of $ \mathcal{M}_{E_{1}}(\mathcal{R}_1 ) $]
		{\label{RM2}
			\includegraphics[height=2cm]{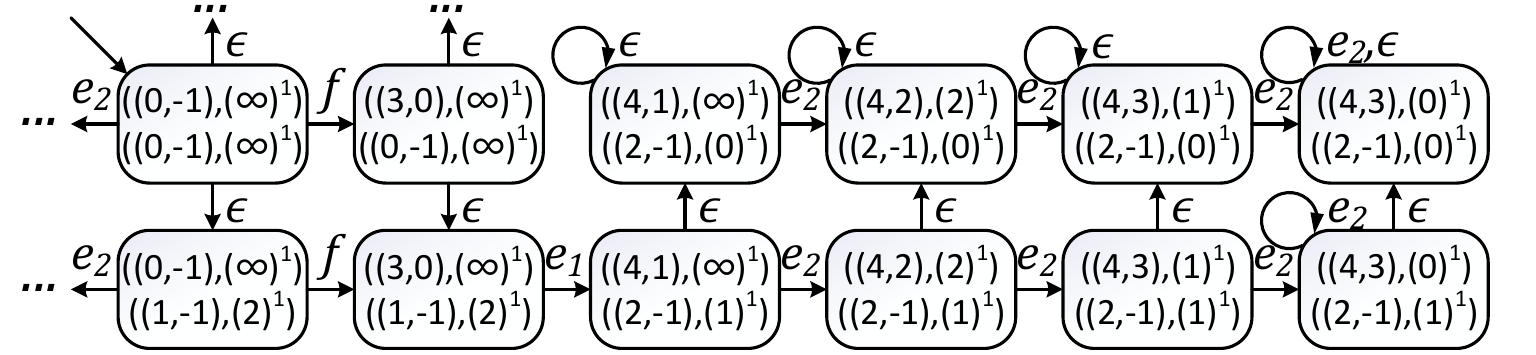}}
		\subfigure[$ \hat{\mathcal{R}}_1 $]
		{\label{DR2}
			\includegraphics[height=2.1cm]{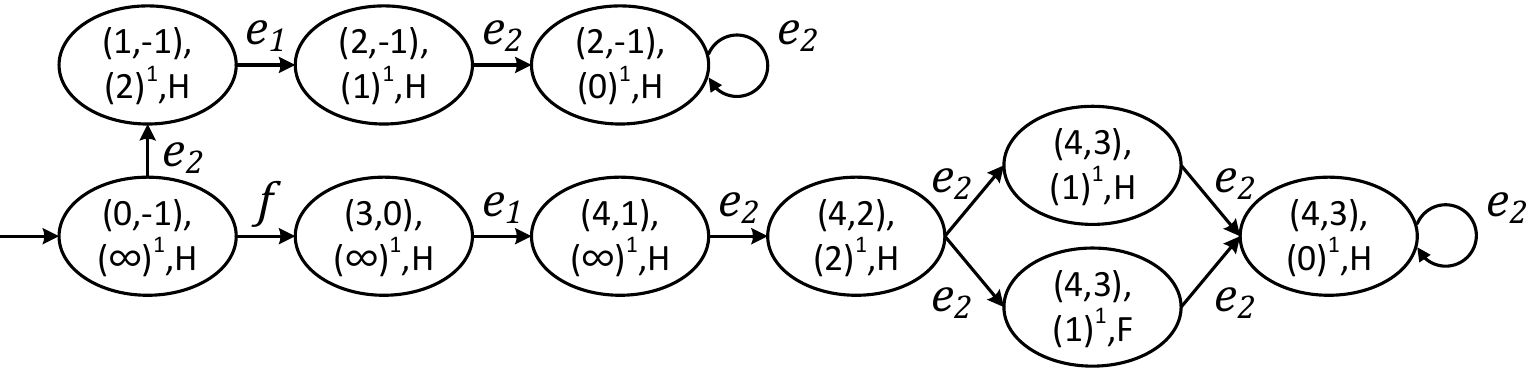}}
		\caption{Part of $ M $-machine $ \mathcal{M}_{E_{1}}(\mathcal{R}_1 ) $ that may include the fault states and the reconstructed automaton $ \hat{\mathcal{R}}_1 $}
	\end{figure}
	
	Using the automaton $ \hat{\mathcal{R}}_i $, we then construct the $ M $-machine $ \mathcal{M}_{E_{k}}(\hat{\mathcal{R}}_i) $ to obtain $ Z_{E_k}^C $ and further run Algorithm \ref{AOS} with $ \hat{\mathcal{R}}_i $ and $ E_k $ to get the augmented delay recorder 
	\begin{equation}\label{Rik} 
	\mathcal{R}_{i,k} = (X_{i,k}, E, \alpha_{i,k}, X_{i,k,0}),
	\end{equation}
	where similar to $ a_i $, we can obtain $ n_k $ and denote the delay value of $ a_k $ as $ |\cdot|_k^{j_k} $ ($ j_k\in \{1,\dots,n_k\} $).
	Then, using $ E_k $, the $ M $-machine 
	\begin{equation}\label{Mik}
	\mathcal{M}_{E_k}(\mathcal{R}_{i,k})=(Z_{i,k}, E \cup \{\epsilon\}, \delta_{i,k}, Z_{i,k,0})
	\end{equation} 
	is obtained, where we still denote $ I_1(z) $ as the first state component, $ I_2(z) $ as the second state component for each state $ z \in Z_{i,k} $.
	Finally, we define a diagnosis function $ \psi : z \rightarrow \{H,F\} $ as follows:
	$\forall z \in Z_{i,k}, $
	\begin{equation}\label{DF}
	\psi(z) = \left\{
	\begin{aligned}
	F , & \quad \text{if } z \in \mathcal{VC}_k; \\
	H , & \quad \text{otherwise};
	\end{aligned}
	\right.
	\end{equation}
	where condition $ z \in \mathcal{VC}_k $ is defined as:
	\begin{equation}\label{vck}
	|I_1(z)|_d \! = \! F, |I_2(z)|_f \! = \! -1, |I_1(z)|_k^{j_k} \! > \! 0, |I_2(z)|_k^{j_k} \! > \! 0.
	\end{equation}
	With a slight abuse of notation,	although the notation $ |\cdot|_d $ is defined for states in $ \hat{\mathcal{R}}_i $, we denote $ |I_1(z)|_d=F $ for $ I_1(z) \in X_{i,k} $ in $ \mathcal{R}_{i,k} $.
	This is possible because $ \mathcal{R}_{i,k} $ is built from $ \hat{\mathcal{R}}_i $, the only difference being that $ X_{i,k} $ contains the delay information of $ a_k $.
	Similarly, we also allow the states in $ X_{i,k} $ to use $ |\cdot|_f $.

	Now we are in the position to verify $ K^T\! $-codiagnosability with the following theorem.
	\begin{theorem}\label{TKT}
		Let $G$ in (\ref{S}) be the system model,  $ E_1 $ and $ E_2 $ be the set of observable events for agents $ a_1 $ and $ a_2 $, $f$ be the fault events, $ \mathcal{M}_{E_k}(\mathcal{R}_{i,k}) $ in (\ref{Mik}) be the $ M $-machine built from the augmented delay recorder $ \mathcal{R}_{i,k} $ in (\ref{Rik}).
		Then, $ \mathcal{L}(G) $ is $ K^{T} $-codiagnosable w.r.t. $f$ if and only if 
		\begin{equation}\label{SI}
		\forall z \in Z_{i,k}, \psi(z) = H.
		\end{equation}
	\end{theorem}
\vspace{0.24em}
	\begin{proof}
		($ \Rightarrow $) 
		By contradiction, let us first suppose that $ \mathcal{L}(G) $ is $ K^{T} $-codiagnosable w.r.t. $f$ while $ \exists z \in Z_{i,k} $, s.t. $ \psi(z) = F $. 
		Then, we have that $ z $ satisfies (\ref{vck}).
		
		First, we consider $ a_k $: combining (\ref{vck}) and (\ref{vci}), we have $ \exists j \in \{1,\dots,n_k\} $, s.t. $ |I_1(z)|_k^j>0 $, $ |I_2(z)|_k^j>0 $, $ |I_1(z)|_f=K $, $ |I_2(z)|_f=-1 $ and $ I_1(z),I_2(z) \in X_{i,k} $.
		
		Since $ |I_1(z)|_k^j>0 $ and $ |I_2(z)|_k^j>0 $, there must be a pair of event strings: $ s^f,s_k^c \in \mathcal{L}(\mathcal{R}_{i,k}): I_1(z) \in \alpha_{i,k}(s^f) \wedge I_2(z) \in \alpha_{i,k}(s_k^c) \wedge P_{E_k}(s^f)=P_{E_k}(s_k^c) $ such that two transitions are marked, where we denote the events of the marked transition as $ e^{f1} $ in $ s^f $ and $ e^{c1} $ in $ s_k^c $.
		Then, we further denote $ s^f = s^{f1}e^{f1}s^{f2} $ and $ s_k^c= s^{c1}e^{c1}s^{c2} $.
		Recalling the \texttt{FORWARD1} procedure in Algorithm \ref{AOS}, we have $ P_{E_o}(s^{f1}e^{f1})=P_{E_o}(s^{c1}e^{c1}) $.
		We now consider the following two cases:
		\begin{itemize}
			\item[i)] If the delay of an event $ e^{f2} $ (or $ e^{c2} $) $ \!\in\! E_o \backslash E_k $ in $ s^{f2} $ ($ s^{c2} $) is recorded, then the system will enter $ I_1(z) $ (or $ I_2(z) $) within $ \lceil \frac{|a_1a_2|}{T} \rceil\!-\!1 $ steps after the occurrence of $ e^{f2} $ (or $ e^{c2} $), that is, $ 0\!<\!|I_1(z)|_k^j \!\leq\! \lceil \frac{|a_1a_2|}{T} \rceil $ (or $ 0\!<\!|I_1(z)|_k^j\!\leq\! \lceil \frac{|a_1a_2|}{T} \rceil $).
			In this case, we further denote $ s^{f2} \!=\! s^{f3}e^{f2}s^{f4} $ (or $ s^{c2} \!=\! s^{c3}e^{c2}s^{c4} $), where $ |s^{f4}|\!<\!\lceil \frac{|a_1a_2|}{T} \rceil $ (or $ |s^{c4}|\!<\!\lceil \frac{|a_1a_2|}{T} \rceil $).
			\item[ii)] If no delay in $ s^f $ (or $ s_k^c $) is recorded, then the system will enter $ I_1(z) $ (or $ I_2(z) $) without any occurrence of the events in $ E_o \backslash E_k $, that is, $ |I_1(z)|_k^j\!=\!\infty $ (or $ |I_1(z)|_k^j\!=\!\infty $).
		\end{itemize}
		Recalling the \texttt{FORWARD1} procedure in Algorithm \ref{AOS}, we know that $ P_{E_o\backslash E_k}(s^{f1}e^{f1})=P_{E_o\backslash E_k}(s^{f1}e^{f1}s^{f3}) $ and $ P_{E_o\backslash E_k}(s^{c1}e^{c1})=P_{E_o\backslash E_k}(s^{c1}e^{c1}s^{c3}) $ in case i), and $ P_{E_o\backslash E_k}(s^{f1}e^{f1})=P_{E_o\backslash E_k}(s^f) $ and $ P_{E_o\backslash E_k}(s^{c1}e^{c1})=P_{E_o\backslash E_k}(s_k^c) $ in case ii).
		Then, we have $ s^{f1}e^{f1} \in \zeta_{E_o \backslash E_k}^{\lceil \frac{|a_1a_2|}{T} \rceil}(s^f) $ and $ s^{c1}e^{c1} \in \zeta_{E_o \backslash E_k}^{\lceil \frac{|a_1a_2|}{T} \rceil}(s_k^c) $, indicating that $ P_{E_o}(\zeta_{E_o \backslash E_k}^{\lceil \frac{|a_1a_2|}{T} \rceil}(s^f)) \cap P_{E_o}(\zeta_{E_o \backslash E_k}^{\lceil \frac{|a_1a_2|}{T} \rceil}(s_k^c)) \neq \emptyset $.
		
		Now we consider $ a_i $: we know that $ \mathcal{R}_{i,k} $ is built with Algorithm \ref{AOS} from $ \hat{\mathcal{R}}_i $ which is reconstructed from the $ M $-machine $ \mathcal{M}_{E_{i}}(\mathcal{R}_i)=(Z_i, E \cup \{\epsilon\}, \delta_i, Z_{i0}) $.
		From $ |I_1(z)|_d=F $, we have the corresponding $ |\hat{x}_i|_d=F $, and further $ z_i $ satisfies $ \exists j' \in \{1,\dots,n_i\} $, s.t. $ |I_1(z_i)|_i^{j'}>0 $, $ |I_2(z_i)|_i^{j'}>0 $, $  |I_1(z_i)|_f=K $ and $ |I_2(z_i)|_f=-1 $, where $ I_1(z_i),I_2(z_i) \in X_{i} $ in $ \mathcal{R}_i = (X_i, E, \alpha_i, X_{i0}) $.
		Then, from $ I_1(z) \in \alpha_{i,k}(s^f) $, we know that $ I_1(z_i) \in \alpha_i(s^f) $.
		Recalling the property of $ M $-machine, $ \exists s_i^c \in \mathcal{L}(\mathcal{R}_i): I_2(z_i) \in \alpha_i(s_i^c) $, s.t. $ P_{E_i}(s_i^c)=P_{E_i}(s^f) $. 
		And with a similar analysis as above, we have $ P_{E_o}(\zeta_{E_o \backslash E_i}^{\lceil \frac{|a_1a_2|}{T} \rceil}(s^f)) \cap P_{E_o}(\zeta_{E_o \backslash E_i}^{\lceil \frac{|a_1a_2|}{T} \rceil}(s_i^c)) \neq \emptyset $.
		
		To sum up, for the event string $ s^f \in \mathcal{L}(\mathcal{R}_{i,k}) = \mathcal{L}(\mathcal{R}_{i}) = \mathcal{L}(G): \Delta(s^f)=K $, $ \forall l \in \{1,2\} $, $ P_{E_l}(s^f) = P_{E_l}(s_l^c) $, $ P_{E_o}(\zeta_{E_o \backslash E_l}^{\lceil \frac{|a_1a_2|}{T} \rceil}(s^f)) \cap P_{E_o}(\zeta_{E_o \backslash E_l}^{\lceil \frac{|a_1a_2|}{T} \rceil}(s_l^c)) \neq \emptyset $ and $ \Delta(s_l^c) = -1 $.
		In other words, $ \mathcal{L}(G) $ is not $ K^{T} $-codiagnosable w.r.t. $f$, resulting in a violation.
		
		($ \Leftarrow $)
		By contradiction, let us suppose that $ \forall z \in Z_{i,k}, \psi(z) = H $ while $ \mathcal{L}(G) $ is not $ K^{T} $-codiagnosable w.r.t. $f$. 
		Then, we have that there exists $ s \in \mathcal{L}(G): \Delta(s) = K $, such that condition (\ref{KTCDC}) is violated for $ a_1 $ and $ a_2 $. 
		
		First, we consider $ a_i \!:\! \exists s_i^c \!\in\! \mathcal{L}(\hat{G})=\mathcal{L}(\mathcal{R}_{i})\!:\! \Delta(s_i^c) \!=\! -1 $, s.t. $ P_{E_i}(s) \!=\! P_{E_i}(s_i^c) $, $ P_{E_o}(\zeta_{E_o \backslash E_i}^{\lceil \frac{|a_1a_2|}{T} \rceil}(s)) \!\cap\! P_{E_o}(\zeta_{E_o \backslash E_i}^{\lceil \frac{|a_1a_2|}{T} \rceil}(s_i^c)) \!\neq\! \emptyset $.
		
		Then, we have that $ \exists e^{f1}, e^{c1} \in E_o\backslash E_i: s=s^{f1}e^{f1}s^{f2}, s_i^c=s^{c1}e^{c1}s^{c2} $, s.t. $ P_{E_o}(s^{f1}e^{f1}) =P_{E_o}(s^{c1}e^{c1}) \in P_{E_o}(\zeta_{E_o \backslash E_i}^{\lceil \frac{|a_1a_2|}{T} \rceil}(s)) \cap P_{E_o}(\zeta_{E_o \backslash E_i}^{\lceil \frac{|a_1a_2|}{T} \rceil}(s_i^c)) $.
		Since $ P_{E_o}(s^{f1}e^{f1})=P_{E_o}(s^{c1}e^{c1}) $, the transitions of $ e^{f1} $ and $ e^{c1} $ in $ s $ and $ s_i^c $ must be marked by an index in line 2 of Algorithm \ref{AOS}.
		Nevertheless, after the first check in lines 3-9, the mark on $ e^{f1} $ and $ e^{c1} $ may be canceled, but an event in $ s^{f2} $ and an event in $ s^{c2} $ will still be marked according to the condition in line 5.
		Hence, we can regard $ e^{f1} $ and $ e^{c1} $ as the marked transition without loss of generality.
		Next, after the second check in lines 11-20, the transition $ e^{f1} $ and $ e^{c1} $ in $ s $ and $ s_i^c $ may not be marked, but there must be another pair of event strings marking $ e^{f1} $ and $ e^{c1} $ according to the conditions in lines 12 and 14.
		Hence, we can regard $ s $ and $ s_i^c $ as the pair of event strings where $ e^{f1} $ and $ e^{c1} $ are marked without loss of generality.
		Now we consider $ j_i \in \{1,\dots,n_i\} $ as the index that marks the transitions of $ e^{f1} $ and $ e^{c1} $ in $ s $ and $ s_i^c $.
		Recalling the \texttt{RECORD} procedure in Algorithm \ref{AOS}, there are two cases to be analysed: 
		\begin{itemize}
			\item[i)] If the delay of an event $ e^{f2} $ (or $ e^{c2} $) after $ e^{f1} $ (or $ e^{c1} $) is recorded, then we can denote $ s^{f2}=s^{f3}e^{f2}s^{f4} $ (or $ s^{c2}=s^{c3}e^{c2}s^{c4} $).
			Here we know $ |s^{f4}|<\lceil \frac{|a_1a_2|}{T} \rceil $ (or $ |s^{c4}|<\lceil \frac{|a_1a_2|}{T} \rceil $), otherwise there will be a violation that $ P_{E_o}(s^{f1}e^{f1}) \notin P_{E_o}(\zeta_{E_o \backslash E_i}^{\lceil \frac{|a_1a_2|}{T} \rceil}(s)) $ (or $ P_{E_o}(s^{c1}e^{c1}) \notin P_{E_o}(\zeta_{E_o \backslash E_i}^{\lceil \frac{|a_1a_2|}{T} \rceil}(s_i^c)) $).
			Since the delay is recorded as $ \lceil \frac{|a_1a_2|}{T} \rceil $ steps, there must be a state $ x_i \in \alpha_i(s): |x_i|_i^{j_i} >0 $ (or $ x_i' \in \alpha_i(s_i^c): |x_i'|_i^{j_i} >0 $).
			\item[ii)] If no delay after $ e^{f1} $ (or $ e^{c1} $) is recorded, then there must be a state $ x_i \in \alpha_i(s): |x_i|_i^{j_i} = \infty $ (or $ x_i' \in \alpha_i(s_i^c): |x_i'|_i^{j_i} = \infty $).
		\end{itemize}
		Since $ P_{E_i}(s) = P_{E_i}(s_i^c) $, we have $ (x_i, x_i') \in \delta_i(s) \subseteq Z_i $ in $ \mathcal{M}_{E_i}(\mathcal{R}_i) $ with $ |x_i|_i^{j_i} >0 $, $ |x_i'|_i^{j_i} >0 $, $ |x_i|_f=K $ and $ |x_i'|_f=-1 $, which means the corresponding state $ \hat{x}_i \in \hat{\alpha}_i(s) $ in $ \hat{\mathcal{R}}_i $, as well as $ x_{i,k} \in \hat{\alpha}_{i,k}(s) $ in $ \mathcal{R}_{i,k} $, satisfies $ |\hat{x}_i|_d = |x_{i,k}|_d = F $.
		
		Now we consider$ \, a_k \!\!:\! \exists s_k^c \!\in\! \mathcal{L}(\hat{G}) \!=\! \mathcal{L}(\mathcal{R}_{i,k})\!\!:\! \Delta(s_k^c)\!=\!-1,\;\! $s.t. $ P_{E_k}(s) \!=\! P_{E_k}(s_k^c), P_{E_o}(\zeta_{E_o \backslash E_k}^{\lceil \frac{|a_1a_2|}{T} \rceil}(s)) \!\cap\! P_{E_o}(\zeta_{E_o \backslash E_k}^{\lceil \frac{|a_1a_2|}{T} \rceil}(s_k^c)) \!\neq\! \emptyset $.
		Then as well, there must be an index $ j_k \!\in\! \{1,\dots,n_k\} $ marking the relevant transitions in $ s $ and $ s_k^c $.
		With a similar analysis as above, we have that $ \exists (x_{i,k},x_{i,k}') \!\in\! \delta_{i,k}(s) \!\subseteq\! Z_{i,k} $ in $ \mathcal{M}_{E_k}(\mathcal{R}_{i,k}) $, s.t. $ x_{i,k} \!\in\! \alpha_{i,k}(s), x_{i,k}' \!\in\! \alpha_{i,k}(s_k^c), |x_{i,k}|_k^{j_k} \!>\! 0, |x_{i,k}'|_k^{j_k} \!>\! 0 $.
		Since $ |x_{i,k}|_d \!=\! F $ and $ |x_{i,k}'|_f\!=\!-1 $, we have $ \psi((x_{i,k},x_{i,k}')) \!=\! F $, resulting in a violation, which completes the proof.
	\end{proof}
	
%	{\color{blue} 
%		For the verification of $ K^T $-codiognosability of $ \mathcal{L}(G) $, we start from the automaton $ \hat{G} $, run Algorithm \ref{AOS} to obtain delay recorder $ \mathcal{R}_i $ 	($ i\in\{1,2\} $), and construct the automaton $ \hat{\mathcal{R}}_i $ derived from the $ M $-machine $ \mathcal{M}_{E_{i}}(\mathcal{R}_i) $.
%		Next, we execute Algorithm \ref{AOS} on $ \hat{\mathcal{R}}_i $ to obtain $ \mathcal{R}_{i,k} $	($ k\in\{1,2\}, k\neq i $)  and build the $ M $-machine $ \mathcal{M}_{E_{k}}(\mathcal{R}_{i,k}) $ whose states are then checked by the diagnosis function $ \psi $.
%	}
	
	\begin{figure}[htbp]
		\centering
		\subfigure[$ \mathcal{R}_{1,2} $]
		{\label{R21}
			\includegraphics[height=2.2cm]{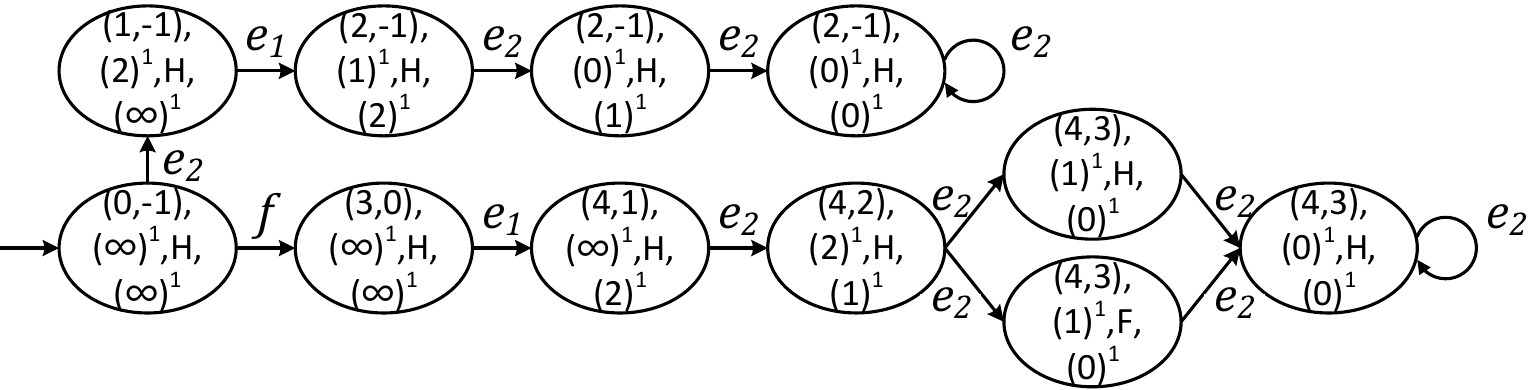}}
		\subfigure[Part of $ \mathcal{M}_{E_{2}}(\mathcal{R}_{1,2}) $]
		{\label{RM1}
			\includegraphics[height=2.8cm]{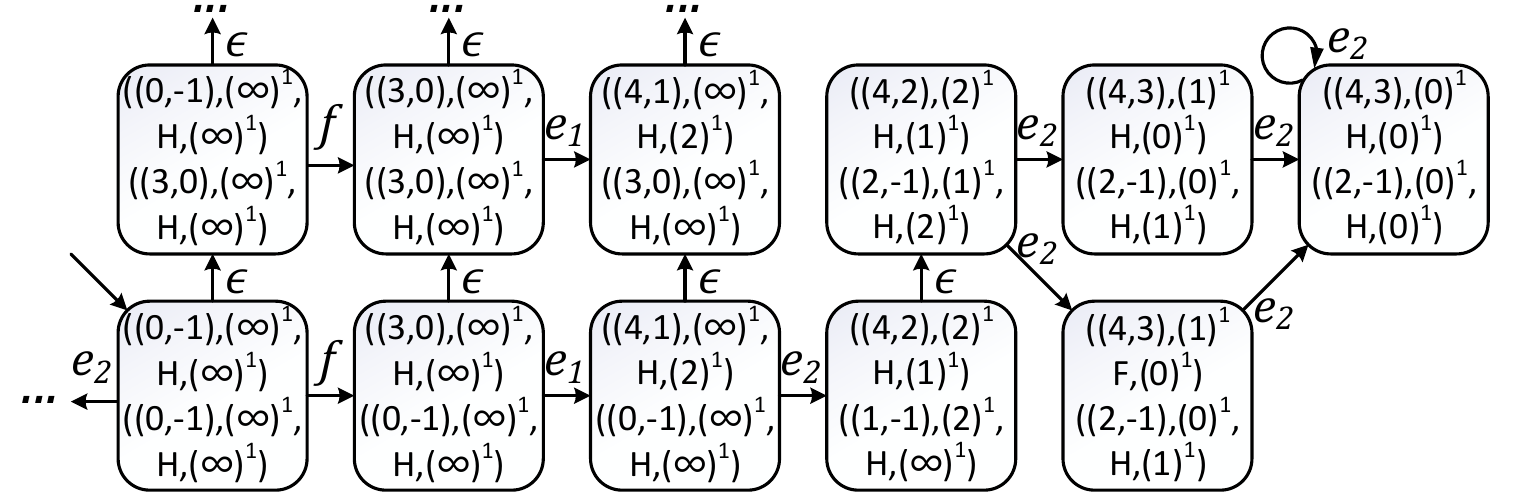}}
		\caption{Augmented delay recorder $ \mathcal{R}_{1,2} $ and part of $ M $-machine $ \mathcal{M}_{E_{2}}(\mathcal{R}_{1,2}) $ that may include the fault state}
	\end{figure}
	
	\begin{example}\label{EKT}(Verifying $ K^T $-codiagnosability).
		Considering the automaton $ \hat{\mathcal{R}}_1 $ in Example \ref{EM}, we build the $ M $-machine $ \mathcal{M}_{E_{2}}(\hat{\mathcal{R}}_1) $ to get $ Z_{E_2}^C=\{(((4,3),(1)^1,F), $ $ ((2,-1),(0)^1,H))\} $.
		Then, we run Algorithm \ref{AOS} with $ \hat{\mathcal{R}}_1 $, $ \lceil \frac{|a_1a_2|}{T} \rceil $, $ Z_{E_2}^C $, $ E_o $ and $ E_2 $ to obtain the augmented delay recorder $ \mathcal{R}_{1,2} $ as shown in Fig. \ref{R21}.
		Finally, the $ M $-machine $ \mathcal{M}_{E_{2}}(\mathcal{R}_{1,2})=(Z_{1,2}, E \cup \{\epsilon\}, \delta_{1,2}, Z_{0}) $ is constructed with $ E_2 $: Fig. \ref{RM1} shows 11 of the 31 states of $ \mathcal{M}_{E_{2}}(\mathcal{R}_{1,2}) $, where the omitted 20 states $ z \in Z_{1,2} $ obviously satisfy $ |I_1(z)|_d=H $ according to (\ref{vck}).
		The only state $ z $ satisfying $ |I_1(z)|_d=F $ is $ (((4,3),(1)^1,F,(0)^1), ((2,-1),(0)^1,H,(1)^1)) $, however, we have $ |((4,3),(1)^1,F,(0)^1)|_2^1=0 $, violating (\ref{vck}).
		Hence, for any $ z \in Z_{1,2} $, we have $ \psi(z)=H $, indicating that $ \mathcal{L}(G) $ is $ K^{T} $-codiagnosable w.r.t. $f$ with $ K=3 $ and $ T=1 $.
		Despite $ a_1 $ failing to diagnose the fault (cf. Example \ref{EKTCD}), the diagnosis task can be fulfilled thanks to $ a_2 $, indicating that distributed diagnosability is possible if and only if at least one of the agents can diagnose the faults with the information received from other agents.
		\hfill \ensuremath{\Box}
	\end{example}

	\section{Conclusion}\label{8}
	
	In this paper, a novel framework has been presented to solve the problem of distributed fault diagnosis in discrete event systems with delays arising from transmission impairments. 
	A new notion of $ K^T $-codiagnosability was proposed that extends the well-known $ K $-codiagnosability to the distributed setting.
	Accordingly, a novel delay recorder structure and a new diagnosis function were proposed to verify $ K^T $-codiagnosability.
	Future work could consider more complex diagnosis problems with transmission impairments, such as distributed dynamic sensor activation.
	
	\bibliographystyle{IEEEtran}
	\bibliography{References}
	
\end{document}